\documentclass[a4paper,11pt]{amsart}
\usepackage{mathtools}
\usepackage{amsthm, amsfonts,amsmath,amscd,amssymb, amsrefs}
\usepackage{tikz-cd}
\usepackage[hidelinks]{hyperref}
\usepackage{stmaryrd}
\usepackage[margin=1.1in]{geometry}
\makeatletter
\@namedef{subjclassname@2020}{\textup{2020} Mathematics Subject Classification}
\makeatother

\title{\textbf{The Picard Group of Vertex Affinoids in the First Drinfeld Covering
}}
\author{James Taylor}
\email{james.taylor@maths.ox.ac.uk}
\address{Mathematical Institute, University of Oxford, Oxford, OX2 6GG, UK}
\date{\today}
\subjclass[2020]{11S37, 14G22 (Primary); 14C22 (Secondary).}
\keywords{Picard group, Drinfeld tower, Deligne-Lusztig variety}

\newcommand{\OO}{\mathcal{O}}

\DeclareMathOperator{\Hom}{Hom}

\DeclareMathOperator{\GL}{GL}
\DeclareMathOperator{\SL}{SL}

\DeclareMathOperator{\Sp}{Sp}
\DeclareMathOperator{\Spec}{Spec}

\DeclareMathOperator{\Proj}{Proj}
\DeclareMathOperator{\Spf}{Spf}

\DeclareMathOperator{\Pic}{Pic}

\DeclareMathOperator{\Cl}{Cl}
\DeclareMathOperator{\etale}{\acute{e}t}

\newcommand{\fn}[3]{#1 : #2 \rightarrow #3}

\newcommand{\bC}{{\mathbb C}}

\newcommand{\bF}{{\mathbb F}}

\newcommand{\bG}{{\mathbb G}}

\newcommand{\bP}{{\mathbb P}}
\newcommand{\bQ}{{\mathbb Q}}

\newcommand{\bZ}{{\mathbb Z}}

\newcommand{\bfW}{{\mathbf W}}

\newcommand{\bfY}{{\mathbf Y}}
\newcommand{\bfZ}{{\mathbf Z}}

\newcommand{\sD}{{\mathcal D}}

\newcommand{\sM}{{\mathcal M}}

\newcommand{\sP}{{\mathcal P}}

\newcommand{\sT}{{\mathcal T}}

\newtheorem{thm}{Theorem}
\numberwithin{thm}{section}
\newtheorem{lemma}[thm]{Lemma} 
\newtheorem{prop}[thm]{Proposition} 
\newtheorem{cor}[thm]{Corollary}

\theoremstyle{definition}
\newtheorem{defn}[thm]{Definition}

\theoremstyle{remark} 
\newtheorem*{remark}{Remark}

\begin{document}
\begin{abstract}
Let $F$ be a finite extension of $\bQ_p$. Let $\Omega$ be the Drinfeld upper half plane, and $\Sigma^1$ the first Drinfeld covering of $\Omega$. We study the affinoid open subset $\Sigma^1_v$ of $\Sigma^1$ above a vertex of the Bruhat-Tits tree for $\GL_2(F)$. Our main result is that $\Pic(\Sigma^1_v)[p] = 0$, which we establish by showing that $\Pic(\bfY)[p] = 0$ for $\bfY$ the Deligne-Lusztig variety of $\SL_2(\bF_q)$.
One formal consequence is a description of the representation $H^1_{\etale}(\Sigma^1_v, \bZ_p(1))$ of $\GL_2(\OO_F)$ as the $p$-adic completion of $\OO(\Sigma^1_v)^\times$.
\end{abstract}

\maketitle

\section{Introduction}

Let $p$ be a prime, $F$ a finite extension of $\bQ_p$, and $K$ the completion of the maximal unramified extension of $F$. Let $\sM_0$ be the disjoint union of $\mathbb{Z}$ copies of $\Omega$, where $\Omega$ is the \emph{Drinfeld upper half plane}: the rigid analytic space over $K$ defined by removing all $F$-rational points from $\bP^{1,\text{an}}_K$. The work of Drinfeld \cite{DRI} implies the existence of a tower of finite \'{e}tale coverings $(\sM_{n})_{n \geq 0}$ of $\sM_0$ equipped with compatible actions of $\GL_2(F)$, which has been shown to realise both the local Langlands and Jacquet-Langlands correspondence in its \'{e}tale cohomology \cite{CAR, HAR, BOY, HARTAY}. On the other hand, there is at present no formulated $p$-adic local Langlands correspondence for $\GL_2(F)$ for general finite extensions $F$. The Drinfeld tower is expected to be of importance in yielding natural representations of $\GL_2(F)$ that should appear in any such correspondence. For example, the geometric $p$-adic \'{e}tale cohomology of the Drinfeld tower has been shown to encode the $p$-adic local Langlands correspondence for $F = \bQ_p$ \cite{CDN1}.

The preimage of the index zero piece $\Omega \hookrightarrow \sM_0$ in the tower $(\sM_n)_{n \geq 0}$ defines a tower $(\Sigma^n)_{n \geq 0}$ of finite \'{e}tale coverings of $\Sigma^0 = \Omega$. The transition morphisms are equivariant for the action of the stabilising subgroup $\GL_2(F)^{+} = \{g \in \GL_2(F) \mid \det(g) \in \OO^\times\}$. Let $\sT$ be the Bruhat-Tits tree for $\GL_2(F)$, $v$ the central vertex of $\sT$, and $r : \Sigma^1 \rightarrow \Omega \rightarrow \sT$ the retraction map. In this paper we study the open affinoid subset $\Sigma^1_{v} \vcentcolon= r^{-1}(v)$ of $\Sigma^1$. This is stable under the action of $\GL_2(\OO_F)$ and after a finite extension of $K$, $\Sigma^1_{v}$ splits up into $q-1$ geometrically connected components, each isomorphic to $\Sp(B)$, where,
\[
B = A[z] / (z^{q+1} - (x^q - x)), \: \text{ for } \: A = K \left\langle x , \frac{1}{x^q - x} \right\rangle.
\]
The group $\GL_2(F)^+$ acts with two orbits on the set of vertices of $\sT$, and one can show that for any vertex $w$ adjacent to $v$, $\Sigma^1_w \cong \Sigma^1_{v}$. As any such $w$ will be in the other orbit from $v$, $\Sigma^1_w \cong \Sigma^1_{v}$ for all vertices $w \in \sT$, and consequently this open subset often determines global properties of $\Sigma^1$. For example, the first de-Rham cohomology $H^1_{\text{dR}}(\Sigma^1)$ as a representation of $\GL_2(F)$ is determined by $H^1_{\text{dR}}(\Sigma^1_v)$ \cite[Thm. 6.1]{JUNDERHAM}.

Our main result is that $\Pic(\Sigma^1_{v})[p] = 0$ (Theorem \ref{mainthm}). The $p$-adic \'{e}tale cohomology groups of Drinfeld spaces are of considerable interest \cite{CDNStein, CDN1, CDNH, CDNInt, ORL, BOS}, and one immediate consequence of Theorem \ref{mainthm} is a description of the $\GL_2(\OO_F)$-representation $H^1_{\etale}(\Sigma^1_v, \bZ_p(1))$, as the $p$-adic completion of $\OO(\Sigma^1_v)^\times$ (Theorem \ref{mainthm3}). This description is very explicit, as the unit group $\OO(\Sigma^1_v)^\times$ has been described by Junger \cite[Thm. 5.1]{JUNEQN}.

Our main interest in Theorem \ref{mainthm} is the following. A precise statement of the $p$-adic local Langlands correspondence is formulated when $F = \bQ_p$ \cite{COL}, and Dospinescu and Le Bras \cite{DLB} have used this to show that for $F = \bQ_p$ and all $n \geq 1$, the representation $\OO(\Sigma^{n})$ is naturally a coadmissible module over $D(G,K)$, the distribution algebra of $G$. 

In an effort to remove the restriction on $F$, Ardakov and Wadsley show in their forthcoming work \cite{AW} using $p$-adic $\sD$-modules that the representation $\OO(\Sigma^1)$ splits up naturally into a direct sum of coadmissible $D(G,K)$-modules. This decomposition contains $\OO(\Omega)$, and all other components are shown to be topologically irreducible $D(G,K)$-modules. The benefits of this approach over that of \cite{DLB}, are that it holds for general field extensions $F$, is purely local, and establishes topological irreducibility. The obvious disadvantage is that it describes $\OO(\Sigma^n)$ only for $n = 1$. One would like to establish similar results for $\OO(\Sigma^n)$ for $n \geq 2$, where the situation is significantly more complicated. This is partially due to the fact that $\Sigma^n \rightarrow \Sigma^{n-1}$ has degree a power of $p$, whearas the degree of $\Sigma^1 \rightarrow \Omega$ is coprime to $p$. The methods of \cite{AW} use the standard result that $\Pic(\Omega) = 0$, and in attempting to transfer these methods to $\OO(\Sigma^2)$, one considers the group $\Pic(\Sigma^1)[p]$ instead.
Almost nothing is known about $\Pic(\Sigma^1)[p]$, which is strongly expected to be non-zero.
Our result that $\Pic(\Sigma^1_v)[p] = 0$ is therefore slightly suprising. It also provides the first steps towards computing $\Pic(\Sigma^1)[p]$ (by choosing an appropriate \v{C}ech cover), and allows one the possibility of using similar methods to \cite{AW} locally.

In order to prove Theorem \ref{mainthm}, we consider the affine curve $\bfY$ defined by,
\[
xy^q - yx^q = 1 ,
\]
over the residue field of $K$, where $\bF_q$ is the residue field of $F$. This curve was first considered by Drinfeld, who showed that all the discrete series representations of $\SL_2(\bF_q)$ can be realised in the cohomology of $\bfY$ \cite[Pref.]{Bon}. Inspired by this, these ideas were generalised to all reductive groups $\bG$ by Deligne and Lusztig in their landmark paper \cite{DEL}. They introduce what are now called \emph{Deligne-Lusztig} varieties, which assign to $\bG(\bF_q)$ and $w \in W$, the Weyl group, a base space $X(w)$ and a finite covering $Y(w)$, and it is in the \'{e}tale cohomology of $Y(w)$ that the cuspidal representations are realised. These are spaces of considerable interest, and the Picard groups of the base spaces $X(w)$ have been considered in \cite{HAN}. Here we consider $\bfY = Y(w)$ in the special case of $\mathbb{G} = \SL_2$, and $w \neq 1$. It would be interesting to study the Picard groups of $Y(w)$ more generally.

\subsubsection*{Acknowledgements}
The author would like to thank Konstantin Ardakov, Damien Junger and the referee for their comments on this paper. This research was financially supported by the EPSRC.

\section{Deligne-Lusztig Curves}

Throughout this section, let $\bF$ be an algebraic field extension of $\bF_q$. We consider the affine curve,
\[
\bfY = \Spec\left(\frac{\bF[x,y]}{xy^q - yx^q = 1}\right),
\]
and its projective closure,
\[
\bfZ= \Proj \left( \frac{\bF[X,Y,Z]}{XY^q - YX^q = Z^{q+1}} \right).
\]
We also consider the projective curve,
\[
\bfW= \Proj \left( \frac{\bF[U,V,W]}{UV^q + VU^q = W^{q+1}}   \right).
\]

We would first like to show that $\Pic(\bfZ)[p] = 0$.

\begin{lemma}\label{minimallemma}
$\bfZ$ is a smooth integral projective curve over $\bF$. Furthermore, if $\bF_{q^4} \subset \bF$, then $\bfW \cong \bfZ$.
\end{lemma}

\begin{proof}
The polynomial $P(X,Y,Z) = Z^{q+1} - (XY^q - YX^q) \in \bF[X,Y,Z]$ is prime, which follows from Eisenstein's criterion for $P \in \bF[X,Y][Z]$, at the prime ideal $(X)$. Therefore $\bfZ$ is integral. Furthermore, $\bfZ$ is smooth, because the system $\partial_X P = \partial_Y P = \partial_Z P = 0$ has no solutions over $\bfZ(\overline{\bF})$.
For the isomorphism, let $\lambda \in \bF_{q^2}$ with $\lambda^{q-1} = -1$, and let $\mu \in \overline{\bF}$ with $\mu^{q+1} = \lambda^q$. The element $\mu$ lies in  $\bF_{q^4}$, as,
\[
\mu^{q^2} = (\lambda^{q})^{q-1}\mu = - \mu,
\]
so,
\[
\mu^{q^4} = (-\mu)^{q^2} = -(- \mu) = \mu.
\]
Then the claimed isomorphism is given by,
\[
U = X, \quad  V = \lambda Y, \quad W = \mu Z.
\]
Indeed,
\begin{align*}
X (\lambda Y)^q + (\lambda Y) X^q &= \lambda^q (X Y^q - YX^q), \\
&= \lambda^q Z^{q+1} = (\mu Z)^{q+1},
\end{align*}
and similarly $U(\lambda^{-1}V)^q - (\lambda^{-1} V) U^q = (\mu^{-1} W)^{q+1}$.
\end{proof}

\begin{prop}\label{picZ}
$\Pic(\bfZ)[p] = 0$.
\end{prop}

\begin{proof}
By Lemma \ref{minimallemma}, $\bfZ_{\overline{\bF}} \cong \bfW_{\overline{\bF}}$, and thus the group $\Pic(\bfZ_{\overline{\bF}})[p] \cong \Pic(\bfW_{\overline{\bF}})[p] \cong J(\overline{\bF})[p]$, where $J$ is the Jacobian of $\bfW$.
$\bfW$ is known as the Hermitian curve, defined by affine equation $w^{q+1} = v^q + v$, and is maximal over $\bF_{q^2}$ \cite[Lem. 6.4.4]{Hen}, hence $J(\overline{\bF})[p] = 0$ by \cite[Cor. 2.5]{TAF}.
Then, because pullback induces an exact sequence $0 \rightarrow \Pic(\bfZ) \rightarrow \Pic(\bfZ_{\overline{\bF}})$ \cite[Tag 0CC5]{STACK}, and $p$-torsion is left exact, $\Pic(\bfZ)[p] = 0$.
\end{proof}

Our next goal is to establish that $\Pic(\bfY)[p] = 0$.

\begin{lemma}
$\bfZ(\overline{\bF}) \setminus \bfY(\overline{\bF})$ consists of the $q+1$ points,
\[
\sP \coloneqq \{(a:b:0) \mid (a:b) \in \bP^1(\bF_q)\}.
\]
Furthermore, $\sP = \bfZ(\bF_q)$.
\end{lemma}

\begin{proof}
If $(a:b:c) \in \bfZ(\overline{\bF})$ with $c =0$, then $b^q a - a^q b = 0$, so $b^q a = a^q b$. If $a \neq 0$, then $\left( \frac{b}{a}\right)^q = \frac{b}{a}$, so $\frac{b}{a} \in \bF_q$, and $(a:b) \in \bP^1(\bF_q)$. Similarly, if $b \neq 0$, $(a:b) \in \bP^1(\bF_q)$. Thus $\bfZ(\overline{\bF}) \setminus \bfY(\overline{\bF})= \sP$. To see $\bfZ(\overline{\bF}) \setminus \bfY(\overline{\bF}) = \bfZ(\bF_q)$, there are no points $(a:b:c) \in \bfZ(\bF_q)$ with $c = 1$, because if so then $1 = ab^q - ba^q = ab - ba = 0$, as $a,b \in \bF_q$.
\end{proof}

Therefore the closed points of $\bfZ \setminus \bfY$ are $\sP$ \cite[Prop. 5.4]{GW}, which we enumerate by $\sP = \{P_0, ... ,P_q\}$. From \cite[Ex. 5.12 (a)]{Wei} we have an exact sequence,
\[
\bZ^{q+1} \rightarrow \Cl(\bfZ) \rightarrow \Cl(\bfY) \rightarrow 0,
\]
where the first map sends,
\[
(m_0, ... ,m_q) \mapsto \sum_{i = 0}^q m_i [P_i],
\]
and the second sends, for $I$ a finite set of closed points of $\bfZ$,
\[
\sum_{P \in I} n_P [P] \mapsto \sum_{P \in I \setminus \sP} n_P [P].
\]

Let $\Gamma = \langle [P_0] , ... ,[P_{q}] \rangle \subset \Cl(\bfZ)$ be the image of $\bZ^{q+1}$ in $\Cl(\bfZ)$. The resulting exact sequence,
\[
0 \rightarrow \Gamma \rightarrow \Cl(\bfZ) \rightarrow \Cl(\bfY) \rightarrow 0,
\]
yields the long exact sequence,
\[
0 \rightarrow \Gamma[p] \rightarrow \Cl(\bfZ)[p] \rightarrow \Cl(\bfY)[p] \rightarrow \Gamma / p\Gamma \rightarrow \Cl(\bfZ)/ p \Cl(\bfZ) \rightarrow \Cl(\bfY)/ p\Cl(\bfY) \rightarrow 0,
\]
from the right derived functors of $\Hom_{\bZ}(\bZ / p\bZ, -)$. Then from Proposition \ref{picZ} and the above discussion we have the following.

\begin{prop}\label{exactseq}
There is an exact sequence
\[
0 \rightarrow \Cl(\bfY)[p] \rightarrow \Gamma / p\Gamma \rightarrow \Cl(\bfZ) / p \Cl(\bfZ),
\]
where the map $\Gamma / p\Gamma \rightarrow \Cl(\bfZ) / p \Cl(\bfZ)$ is that induced by the inclusion $\Gamma \hookrightarrow \Cl(\bfZ)$.
\end{prop}

\begin{remark}
We note that if $\bfZ \setminus \bfY$ contained exactly one degree $1$ closed point $Q$, then we could establish that $\Pic(\bfY)[p] = 0$ almost immediately in the following way. In the exact sequence,
\[
\bZ \rightarrow \Cl(\bfZ) \rightarrow \Cl(\bfY) \rightarrow 0,
\]
the map $\bZ \rightarrow \Cl(\bfZ)$ is actually injective and split by the degree homomorphism, hence $\Cl(\bfZ) \cong \bZ \times \Cl(\bfY)$ so,
\[
0 = \Cl(\bfZ)[p] \cong \bZ[p] \times \Cl(\bfY)[p] = \Cl(\bfY)[p].
\]
In particular, this can be applied to show that the class groups of affine dehomogenisations of $\bfZ$ with respect to both $X$ and $Y$ both have no $p$-torsion.
\end{remark}

We want to show that $\Cl(\bfY)[p] = 0$, and so in light of Proposition \ref{exactseq}, we want to show that,
\[
\Gamma / p\Gamma \rightarrow \Cl(\bfZ) / p \Cl(\bfZ),
\]
is injective. In order to do so, we now examine the structure of $\Gamma$. First we compute the principal divisors of some rational functions on $\bfZ$.

\begin{defn}
For $(a:b) \in \bP^1(\bF_q)$, we let $P_{(a:b)}$ be the closed point of $\bfZ$ defined by $(a:b:0) \in \bP^1(\overline{\bF})$.
\end{defn}

\begin{lemma}\label{valuations}
Let $(a:b), (c:d) \in \bP^1(\bF_q)$ with $(a:b) \neq (c:d)$. Then the rational function,
\[
f \coloneqq \frac{bX - aY}{dX - cY},
\]
has associated principal divisor,
\[
(f) = (q+1)[P_{(a:b)}] - (q+1)[P_{(c:d)}].
\]
\end{lemma}

\begin{proof}
Consider the morphism $\zeta : \bfZ \rightarrow \bP^1$ corresponding to the extension of function fields $\bF(\bP^1) \rightarrow \bF(\bfZ)$, which sends,
\[
\frac{S}{T} \mapsto \frac{bX - aY}{dX - cY},
\]
where $\bP^1 = \Proj(\bF[S,T])$, and $\bF(\bP^1) = \bF(S/T)$.
On $\overline{\bF}$-points, $\zeta : \bfZ \rightarrow \bP^1$ is given by,
\[
\zeta(x:y:z) = (bx - ay : dx - cy).
\]
This extension $\bF(\bP^1) \rightarrow \bF(\bfZ)$ has degree $q+1$ because it differs by an automorphism of $\bP^1$ from the extension $\bF(\bP^1) \rightarrow \bF(\bfZ)$, defined by,
\[
\frac{S}{T} \mapsto \frac{X}{Y},
\]
which clearly has degree $q+1$. Let $Q_0, Q_{\infty}$ be the closed points of $\bP^1$ defined by $(0:1), (1:0) \in \bP^1(\overline{\bF})$ respectively. By \cite[Cor. 3.9]{QL}, we have that,
\[
(f) = \zeta^*(S/T) = \zeta^*([Q_0]) - \zeta^*([Q_{\infty}]),
\]
and $\deg(\zeta^*([Q_0])) = \deg(\zeta^*([Q_{\infty}])) = [\bF(\bP^1):\bF(\bfZ)] = q+1$. But $\zeta^*([Q_0])$ is some integer multiple of $[P_{(a:b)}]$ and $\zeta^*([Q_{\infty}])$ some integer multiple of $[P_{(c:d)}]$, hence,
\[
(f) = (q+1)[P_{(a:b)}] - (q+1)[P_{(c:d)}]. \qedhere
\]
\end{proof}

Let $\Gamma^0 \subset \Gamma$ be the degree $0$ subgroup of $\Gamma$, and $\Cl^0(\bfZ) \subset \Cl(\bfZ)$ the degree $0$ subgroup of $\Cl(\bfZ)$.

\begin{lemma}\label{surjhom}
The function $\fn{\phi}{\bZ \times (\bZ / (q+1) \bZ)^q}{\Gamma}$,
\[
\phi: (n_0, ... ,n_q)\mapsto n_0 [P_0] + n_1([P_1] - [P_0]) + ... + n_q([P_q] - [P_0])),
\]
is a surjective homomorphism. In particular, $\Gamma^0$ is a quotient of $(\bZ / (q+1) \bZ)^q$.
\end{lemma}

\begin{proof}
For each $P_k \in \sP$, we can write $P_k = P_{(a_k:b_k)}$ for some $a_k, b_k \in \bF_q$. For each $0 \leq i \neq j \leq q$, consider the rational function,
\[
f = \frac{b_i X - a_iY}{b_j X - a_j Y}.
\]
Taking the divisor of $f$,
\[
0 = (f) = (q+1)[P_i] - (q+1)[P_j],
\]
in $\Gamma$, by Lemma \ref{valuations}. Therefore, $\phi$ is a well-defined homomorphism, which is surjective because $\{[P_0], ... ,[P_q]\}$ generate $\Gamma$. Finally, as $\Gamma^0=\langle [P_1] - [P_0] , ... ,[P_q] - [P_0] \rangle$, then $\Gamma^0$ is a quotient of $(\bZ / (q+1) \bZ)^q$.
\end{proof}

We are finally in a position to prove the main result of this section.

\begin{thm}\label{picY}
$\Pic(\bfY)[p] = 0$.
\end{thm}

\begin{proof}
We can split the degree homomorphism with $[P_0]$, as $[P_0]$ has degree $1$ and $\langle [P_0] \rangle$ is free \cite[Ex. 5.12 (b)]{Wei}. Then,
\begin{align*}
\psi : \Cl(\bfZ) &\rightarrow \Cl^0(\bfZ) \times \bZ, \\
Q &\mapsto (Q - \deg(Q) [P_0], \deg(Q)),
\end{align*}
is an isomorphism, which restricts to,
\[
\Gamma \cong \Gamma^0 \times \bZ.
\]
We then obtain the following commutative diagram,
\[\begin{tikzcd}
	{\frac{\Gamma}{p\Gamma}} & {\frac{\Gamma^0 \times \bZ}{p(\Gamma^0 \times \bZ)}} & {\frac{\Gamma^0}{p\Gamma^0} \times \frac{\bZ}{p\bZ}} \\
	{\frac{\Cl(\bfZ)}{ p\Cl(\bfZ)}} & {\frac{\Cl^0(\bfZ) \times \bZ}{p(\Cl^0(\bfZ) \times \bZ)}} & {\frac{\Cl^0(\bfZ)}{p\Cl^0(\bfZ)} \times \frac{\bZ}{p\bZ}}
	\arrow["\sim",from=1-1, to=1-2]
	\arrow["\sim",from=1-2, to=1-3]
	\arrow[from=1-2, to=2-2]
	\arrow[from=1-3, to=2-3]
	\arrow["\sim",from=2-2, to=2-3]
	\arrow["\sim",from=2-1, to=2-2]
	\arrow[from=1-1, to=2-1].
\end{tikzcd}\]

Here, the vertical maps are induced from the inclusions of $\Gamma$ into $\Cl(\bfZ)$ and of $\Gamma^0$ into $\Cl^0(\bfZ)$, the left horizontal maps are induced by $\psi$, and the right horizontal maps are the standard identifications.

Now, by Lemma \ref{surjhom}, $\Gamma^0$ is a quotient of $(\bZ / (q+1)\bZ)^{q}$, thus $\Gamma^0 / p \Gamma^0 = 0$. Consequently, $\Gamma / p\Gamma \rightarrow \Cl(\bfZ) / p \Cl(\bfZ)$ is an injection. Therefore, $\Cl(\bfY)[p] = 0$, by the exact sequence of Proposition \ref{exactseq}.
\end{proof}

\section{Rigid Curves}\label{affinoid}

Let $F$ be a finite extension of $\bQ_p$ with uniformiser $\pi$ and residue field $\bF_q$. Let $K$ be a complete field extension of $F$ with residue field $\bF$, such that $\bF$ is an algebraic extension of $\bF_q$. Let $R$ be the ring of integers of $K$ and $\varpi \in K$ an element with $0 < |\varpi| < 1$.

Let $A$ be the affinoid algebra,
\[
A = K \left\langle x , \frac{1}{x^q - x} \right\rangle,
\]
for which the associated rigid space $\Sp(A)$ has admissible formal model $\Spf(A_0)$, where,
\[
A_0 = R \left\langle x , \frac{1}{x^q - x} \right\rangle.
\]

Let $u \coloneqq x^q - x \in A_0^\times \subset A^\times$, and let $B$ be the affinoid algebra,
\[
B \coloneqq  A[z] / (z^{q+1} - u).
\]
Consider the ring extension,
\[
B_0 \coloneqq  A_0[z] / (z^{q+1} - u).
\]
$B_0$ is $\varpi$-torsion free, and the natural map,
\[
B_0 = A_0[z] / (z^{q+1} - u) \rightarrow \left. R \left\langle x , \frac{1}{x^q - x}, z \right\rangle \middle/ \left( z^{q+1} - u \right) \right.,
\]
is an isomorphism (because $u \in A_0^\times$ is a unit), hence $B_0$ is an admissible $R$-algebra. The special fibre of $\Spf(B_0)$ is,
\[
\Spec(B_0 \otimes_{R} \bF) = \Spec \left( \bF\left[y, 1 / v, t \right] / (t^{q+1} - v)  \right),
\]
where $v = y^q - y$, and the generic fibre of $\Spf(B_0)$ is $\Sp(B_0 \otimes_R K) = \Sp(B)$.

\begin{lemma}\label{comppicgroups}
$\Pic(\Sp(B)) \cong \Pic(\bfY)$.
\end{lemma}
\begin{proof}
First note that there is an isomorphism of $\bF$-algebras,
\[
\bF[r,s] / (rs^q - sr^q - 1) \xrightarrow{\sim} \bF\left[y, 1/v, t \right] / (t^{q+1} - v),
\]
given by $r \mapsto 1/t$, $s \mapsto y / t$, with inverse $y \mapsto s/r$, $t \mapsto 1/r$. Thus $\Spec(B_0 \otimes_R \bF) \cong \bfY$, and $\Spf(B_0)$ is a smooth admissible formal model of $\Sp(B)$. Therefore by \cite[Lem. 3.6]{HEUER}, the natural maps,
\[
\Pic(\Sp(B)) \xleftarrow{\sim} \Pic(\Spf(B_0)) \xrightarrow{\sim} \Pic(\Spec(B_0 \otimes_R \bF)),
\]
are isomorphisms and we're done.
\end{proof}

We can now state our main results. If $K$ contains $\breve{F}$ the completion of the maximal unramified extension of $F$, then we can consider the rigid analytic space $\Sigma^1$ defined over any such $K$. For an overview of the construction and properties of $\Sigma^1$ see \cite[\S 2]{JUNEQN}. If $v \in \sT$ is the central vertex of the Bruhat-Tits tree, then the open affinoid subset $\Sigma^1_{v} \vcentcolon= r^{-1}(v) \subset \Sigma^1$ has coordinate ring isomorphic to,
\[
\OO(\Sigma^1_v) \cong A[z] / (z^{q^2 - 1} - (\pi u^{q-1})),
\]
by \cite[Thm. 2.7]{JUNEQN}.

Let $\omega$ be a primitive $(q^2 - 1)$st root of $\pi$ in $\overline{F}$. From now on we strengthen our assumption on the complete field extension $K$ of $F$ and assume that,
\[
\textit{$K$ contains $\breve{F}(\omega)$ and $\bF$ is an algebraic extension of $\bF_q$.}
\]
We note that this forces $\bF$ to be an algebraic closure of $\bF_q$, and that this assumption holds for any complete field extension $K$ of $\breve{F}(\omega)$ which is contained in $\bC_p$.

\begin{thm}\label{mainthm}
$\Pic(\Sigma^1_v)[p] = 0$.
\end{thm}

\begin{proof}
Because $K$ contains $\omega$,
\[
\OO(\Sigma^1_{v}) \cong B^{q-1},
\]
and therefore,
\[
\Pic(\Sigma^1_v) \cong \Pic (\Sp(B^{q-1})) = \Pic(\Sp(B))^{q-1} \cong \Pic(\bfY)^{q-1},
\]
by Lemma \ref{comppicgroups}. But then $\Pic(\Sigma^1_v)[p] \cong \Pic(\bfY)[p]^{q-1}$, which is zero by Theorem \ref{picY}.
\end{proof}

Recall that $\Sigma^1_v = r^{-1}(v)$ is the pre-image of $v$, the central vertex of the Bruhat-Tits tree. The vertex $v$ is fixed by $\GL_2(\OO_F)$, and because $r$ is equivariant, $\GL_2(\OO_F)$ acts on $\Sigma^1_v$.

\begin{cor}\label{mainthm2}
The natural map,
\[
\OO(\Sigma^1_v)^{\times} / \OO(\Sigma^1_v)^{\times p^n} \rightarrow H^1_{\etale}(\Sigma^1_v, \mu_{p^n}),
\]
arising from the Kummer exact sequence is an isomorphism of $\GL_2(\OO_F)$-modules.
\end{cor}

\begin{proof}
Because $K$ has characteristic $0$, we can consider the Kummer exact sequence for rigid analytic spaces \cite[Sect. 3.2]{PJ}. Then the result follows from Theorem \ref{mainthm} after taking the long exact sequence in \'{e}tale cohomology, using that $\Pic(\Sigma^1_v) \cong H^1_{\etale}(\Sigma^1_v, \bG_m)$ \cite[Prop. 3.2.4]{PJ}.
\end{proof}

As a consequence, we may now compute $H^1_{\etale}(\Sigma^1_v, \bZ_p(1))$ as the $p$-adic completion of $\OO(\Sigma^1_v)^\times$. This is completely explicit, as the group $\OO(\Sigma^1_v)^{\times}$ has been computed by Junger \cite[Thm. 5.1]{JUNEQN}.

\begin{thm}\label{mainthm3}
There is an isomorphism of $\bZ_p$-linear representations of $\GL_2(\OO_F)$,
\[
H^1_{\etale}(\Sigma^1_v, \bZ_p(1)) \cong \varprojlim_{n \geq 1} \OO(\Sigma^1_v)^{\times} / \OO(\Sigma^1_v)^{\times p^n}.
\]
\end{thm}

\begin{proof}
For all $n \geq 1$ the diagram,
\[\begin{tikzcd}
	\OO(\Sigma^1_v)^{\times} / \OO(\Sigma^1_v)^{\times p^{n+1}} & H^1_{\etale}(\Sigma^1_v, \mu_{p^{n+1}}) \\
	\OO(\Sigma^1_v)^{\times} / \OO(\Sigma^1_v)^{\times p^n} & H^1_{\etale}(\Sigma^1_v, \mu_{p^n})
	\arrow[from=1-1, to=2-1]
	\arrow[from=2-1, to=2-2]
	\arrow[from=1-2, to=2-2]
	\arrow[from=1-1, to=1-2]
\end{tikzcd}\]
commutes. Then by the definition of $H^1_{\etale}(\Sigma^1_v, \bZ_p(1))$ and Corollary \ref{mainthm2},
\[
H^1_{\etale}(\Sigma^1_v, \bZ_p(1)) = \varprojlim_{n \geq 1} H^1_{\etale}(\Sigma^1_v, \mu_{p^n}) \xleftarrow{\sim} \varprojlim_{n \geq 1} \OO(\Sigma^1_v)^{\times} / \OO(\Sigma^1_v)^{\times p^n}. \qedhere
\]
\end{proof}

\bibliography{biblio}{}
\bibliographystyle{plain}
\end{document}